\documentclass [a4]{article}
\author{Kohei Tanaka \footnote{Department of Mathematical Sciences, Shinshu University, Matsumoto, 390-8621, Japan}}

\title {A model structure on the category of small categories for coverings}

\usepackage{wrapfig}
\usepackage[all]{xy}
\usepackage{amsmath,amssymb}
\usepackage{amsthm}
\usepackage{color}

\newtheorem{theorem}{Theorem}[section]
\theoremstyle{definition}
\newtheorem{example}[theorem]{Example}
\newtheorem{definition}[theorem]{Definition}
\newtheorem{notation}[theorem]{Notation}

\theoremstyle{result}
\newtheorem{proposition}[theorem]{Proposition}
\newtheorem{lemma}[theorem]{Lemma}
\newtheorem{corollary}[theorem]{Corollary}

\newtheorem{main1}{Main theorem}

\begin{document}

\maketitle

A{\footnotesize BSTRACT}. 
We consider a model structure on the category of small categories,  
which is intimately related to the notion of coverings and fundamental groups of small categories.
Fibrant objects coincide with groupoids, 
and the fibrant replacement is the groupoidification.
\\

\textit{Mathematics Subject Classification} :  18G55, 55U35

\textit{Key words and phrases} : Model categories, Small categories, Coverings


\section{Introduction}

The category $\mathbf{Cat}$ of small categories has 
a couple of interesting model structures.
One of them is introduced by Joyal and Tierney \cite{JT91}, \cite{Rez00}, which is Quillen adjoint to the 
Anderson model structure \cite{And78}, \cite{CGT04} 
 on the category $\mathbf{Grd}$ of groupoids. 
On the other hand,  Thomason found another model structure \cite{Tho80}
which is Quillen equivalent to the Kan model structure \cite{Qui67},\cite{Hov99} 
on the category $\mathbf{SSet}$ of simplicial sets and the Quillen model structure \cite{Qui67}, \cite{Hov99} 
on the category $\mathbf{Space}$ of topological spaces. 
These model categories are related to each other by the following functors
\[ 
\xymatrix{
\mathbf{Grd} \ar@<1ex>[r]^{i} & \mathbf{Cat}  \ar@<1ex>[l]^{\pi} \ar@<1ex>[r]^{N}   &\mathbf{SSet} \ar@<1ex>[r]^{|-|} \ar@<1ex>[l]^{c}  & \mathbf{Space} \ar@<1ex>[l]^{S} 
}
\]
where $i$, $\pi$, $N$, $c$, $|-|$ and $S$ are the inclusion, the groupoidification, 
the nerve, the categorization, the realization and  
the singular simplicial set functor, respectively. In \cite{Qui68}, 
Quillen shows that Serre fibrations in $\mathbf{Space}$ are related to 
Kan fibrations in $\mathbf{SSet}$ by $|-|$ and $S$.
Similarly, Gabriel and Zisman define coverings in $\mathbf{SSet}$, 
and show that coverings in $\mathbf{Space}$ are related to 
coverings in $\mathbf{SSet}$ by $|-|$ and $S$ in \cite{GZ67}.
On the other hand, they also define coverings of groupoids, 
and show that coverings in $\mathbf{Space}$ are related to 
coverings in $\mathbf{Grd}$ 
by the fundamental groupoid functor in \cite{May99}, \cite{GZ67}.

Now, we consider coverings in $\mathbf{Cat}$ related to coverings in 
$\mathbf{SSet}$ and $\mathbf{Grd}$ by the above functors. 
The aim of this article is to show that we can treat coverings in $\mathbf{Cat}$ in terms of model categories.

\begin{main1}[Theorem \ref{them}] 
The category of small categories becomes a model category with
weak equivalences the weak $1$-equivalences,
cofibrations the injection on the set of objects,
fibrations the fibered and cofibered in groupoids.

\end{main1}

We call the above model structure the ``$1$-type model structure'' and 
denote the category of small categories with the model structure by $\mathbf{Cat}_{1}$.
This is the left Bousfield localization \cite{Hir03} of the Joyal-Tierney model structure on $\mathbf{Cat}$. 
The $1$-type model structure on $\mathbf{SSet}$ and $\mathbf{Space}$ 
is already described in \cite{DP95}, and 
these are equivalent to $\mathbf{Cat}_{1}$ as model categories.
We show that $\mathbf{Cat}_{1}$ is related to the notion of 
coverings in $\mathbf{Cat}$ and the groupoidification.

\begin{main1}
The model category $\mathbf{Cat}_{1}$ has the following properties.
\begin{enumerate}
\item An object is fibrant if and only if it is a groupoid [Corollary \ref{fibrant}].
\item A functor is a covering in the category of small categories if and only if it is a fibration with discrete fibers [Proposition \ref{discrete}].
\item Universal covers in $\mathbf{Cat}$ and the groupoidification functor can be described in terms of the functorial factorization [Corollary \ref{poi}, Corollary \ref{un}].
\end{enumerate}

\end{main1}

This paper is organized as follows. In section $2$, we study the notion of fundamental groups and coverings in $\mathbf{Cat}$. 
We construct a Galois-type correspondence between them, namely, 
subgroups of the fundamental group of $C$ are classified by coverings over $C$.

In section $3$, we introduce $\mathbf{Cat}_{1}$ as the left Bousfield localization 
of the Joyal-Tieney model structure.
We show that a covering is a special case of fibration in $\mathbf{Cat}_{1}$ 
and construct a functorial factorization, directly.
Universal covers in $\mathbf{Cat}$ and the groupoidification are described in terms of the factorization.
Finally, we investigate the relations between $\mathbf{Cat}_{1}$ and other model categories.

\begin{notation}\label{notation}
We use the following notations for categories, 
\begin{enumerate}
\item $\phi$ is the empty category,
\item $*$ is the category with a unique object $*$ and the only identity morphism,
\item $[n]$ is the poset $0<1<2 < \cdots <n$ regarded as a category 
$0 \longrightarrow 1 \longrightarrow 2 \longrightarrow \cdots \longrightarrow n,$
\item $S^{0}$ is the category with two objects $\{0,1\}$ and the only identity morphisms,
\item $S^{1}$ is the category with two objects and having 
two parallel morphisms between them $0 \rightrightarrows 1$,
\item $I_{n}$ is the category $0 \longrightarrow 1 \longleftarrow 2 \longrightarrow \cdots \longleftarrow n$ (case $n$ even),
\item $CS^{1}$ is the category consisting of three objects $0 \longrightarrow 1 \rightrightarrows 2$ where $0$ is the initial object,
\item $S^{\infty}$ is the simply connected groupoid with two objects $0 \leftrightarrows 1$.
\end{enumerate}

We write the set of objects by $C_{0}$ and the hom-set from $a$ to $b$ by $C(a,b)$ for a category $C$.
A pointed category $(C,x)$ is a pair of a small category $C$ and an object $x$ of $C$.
Also we use the next notations for set of functors,
\begin{enumerate}
\item $K=\{k : * \longrightarrow S^{\infty}\}$ where $k(*)=0$,
\item $I = \{ \phi \longrightarrow * \ ,\  i : S^{0} \longrightarrow [1] \ ,\ i' : S^{1} \longrightarrow [1]\}$ where 
both $i$ and $i'$ are the identity maps on the set of objects,
\item $J_{1}=\{ j_{1} : * \longrightarrow [1] \ ,\ j_{1}^{op} : * \longrightarrow [1]^{op}\}$ where $j_{1}(*)=0$,
\item $J_{2} = \{j_{2} : I_{2} \longrightarrow [2] \ ,\ j_{2}^{op} : I_{2}^{op} \longrightarrow [2]^{op}\}$ where 
$j_{2}(0)=0$, $j_{2}(1)=2$ and $j_{2}(2)=1$. 
\item $J_{3}=\{ j_{3} : CS^{1} \longrightarrow [2]\ ,\ j_{3}^{op} : (CS^{1})^{op} \longrightarrow [2]^{op}\}$ where $j_{3}$ 
is the identity map on the set of objects,
\item $J = J_{1} \cup J_{2} \cup J_{3}$.
\end{enumerate}
\end{notation}


\section{Fundamental groups and coverings of small categories}


\subsection{Fundamental groups of small categories}

Minian defined the fundamental group $\pi_{1}(C,x)$ for a pointed category $(C,x)$ as 
the colimit of the set of strong homotopy classes of functors from interval categories $I_{n}$ to $C$ for $n \geqq 0$ \cite{Min02}.
It is the endmorphism group of the groupoidification of $C$. 
The groupoidification is an operation to add formal inverses to all morphisms in a small category. 
A concrete construction is the following.

\begin{definition}\label{pi}
For a small category $C$, let $(C^{I};a,b)$ be the set of functors
$$(C^{I};a,b) =\{\alpha : I_{n} \longrightarrow C \ |\ \alpha(0)=a , \alpha(n)=b, n \geqq 0\}$$
for $a,b \in C_{0}$, and define a relation $\sim$ on $(C^{I};a,b)$ by 

\begin{enumerate}
\item $(c \stackrel{f}{\longrightarrow} d \stackrel{f}{\longleftarrow} c) \sim c \sim (c \stackrel{f}{\longleftarrow} d \stackrel{f}{\longrightarrow} c)$,
\item $(c \stackrel{f}{\longrightarrow} d \stackrel{=}{\longleftarrow} d \stackrel{g}{\longrightarrow} e) \sim (c \stackrel{g \circ f}{\longrightarrow} e)$, $(c \stackrel{f}{\longleftarrow} d \stackrel{=}{\longrightarrow} d \stackrel{g}{\longleftarrow} e) \sim (c \stackrel{f \circ g}{\longleftarrow} e)$, 
\item $(c \stackrel{f}{\longrightarrow} b \stackrel{=}{\longleftarrow} b) \sim (c \stackrel{f}{\longrightarrow} b)$.
\end{enumerate}

Define a small category $\pi (C)$  by
$\pi(C)_{0} = C_{0}$ and $\pi(C)(a,b)=(C^{I};a,b)/\sim$.
The composition is given by concatenation. All of the morphisms
are invertible, hence $\pi(C)$ is a groupoid and it gives a functor $\pi : \mathbf{Cat} \longrightarrow \mathbf{Grd}$.

For a pointed small category $(C,x)$, define the fundamental group as the endmorphism group of $\pi(C)$
$$\pi_{1}(C,x) := \pi (C)(x,x).$$
It is easy to show that the relation $\sim$ on  $(C^{I};a,b)$ is equal to the one defined by strong homotopy in \cite{Min02}.
And the fundamental group coincides with the Minian's. 
We say that $C$ is connected if $\pi_{0}(C)=*$, and 
simply connected if it is connected and $\pi_{1}(C,x)$ is trivial for any $x \in C_{0}$.

If $C$ is connected,  it is obvious that $\pi_{1}(C,c) \cong \pi_{1}(C,d)$ for all $c,d \in C_{0}$, in which case
 we write $\pi_{1}(C,x)$ simply as $\pi_{1}(C)$.

\end{definition}

The groupoid $\pi(C)$ is called the groupoidification of $C$. It is the minimal groupoid containing $C$ as a subcategory.

\begin{proposition}\label{Grd}
The functor $\pi$ is left adjoint to the inclusion $\mathbf{Grd} \longrightarrow \mathbf{Cat}.$
\begin{proof}
Let $C$ be a small category and $G$ a groupoid. The canonical inclusion $C \longrightarrow \pi(C)$ 
induces a natural isomorphism $\mathbf{Grd}(\pi(C),G) \longrightarrow \mathbf{Cat}(C,G)$.
\end{proof}
\end{proposition}

\begin{theorem}[\cite{Min02}]\label{min}
Let $(C,x)$ be a pointed category, then there is an isomorphism
$$\pi_{1}(C,x) \cong \pi_{1}(BC,x)$$
where $BC$ is the classifying space of $C$.
\end{theorem}

By the above theorem, $\pi_{1}(C,*)$ can be studied by using homotopy theoretic properties of $BC$. 
However, we can describe $\pi_{1}(C,*)$ in terms of morphisms in $C$ in certain cases.

\begin{proposition}
If the base point $*$ of $C$ is an initial or a terminal object, then $\pi_{1}(C,*)$ is trivial.
\begin{proof}
Let $*$ be an initial object and consider a sequence
$$(*=c_{0} \stackrel{f_{1}}{\longrightarrow} c_{1} \stackrel{f_{2}}{\longleftarrow} c_{2} \stackrel{f_{2}}{\longrightarrow} \cdots \longleftarrow c_{n}=*),$$
then there exists a unique morphism $\alpha_{2} : *=c_{0} \longrightarrow c_{2}$.
On the other hand, the set $C(c_{0},c_{1})$ consists of the single point  
$f_{2} \circ \alpha_{2}=f_{1}$. Therefore,
\begin{equation}
\begin{split} 
&(*=c_{0} \stackrel{\alpha_{2}}{\longrightarrow} c_{2} \stackrel{=}{\longleftarrow} c_{2} \stackrel{f_{2}}{\longrightarrow} c_{1} \stackrel{f_{2}}{\longleftarrow} c_{2} \longrightarrow  \cdots \longleftarrow c_{n}=*) \notag \\
=&(*=c_{0} \stackrel{f_{3} \circ \alpha_{2}}{\longrightarrow} c_{3} \longleftarrow  \cdots \longleftarrow c_{n}=*).
\end{split}
\end{equation}
By iterating this operation, 
the above sequence can be shown to be equivalent to $* \stackrel{=}{\longrightarrow} *$, thus $\pi_{1}(C,*)=1$.
Similarly, we can prove that $\pi_{1}(C,*)=1$ if $*$ is a terminal object.
\end{proof}
\end{proposition}

\begin{example}
Recall the category $S^{1}$ given in Notation \ref{notation}.
It consists of two objects $0,1$ 
and two parallel morphisms $f,g$ and identity morphisms
\[
\xymatrix
{
0 \ar@<1ex>[r]^{f}  \ar@<-1ex>[r]_{g} & 1.\\
}
\]
Thus $\pi_{1}(S^{1})$ is generated by $(0 \stackrel{f}{\longrightarrow} 1 \stackrel{g}{\longleftarrow} 0)$, 
and  $\pi_{1}(S^{1}) \cong \mathbb{Z}$.
\end{example}

\begin{example}\label{group}{\rm
Let $G$ be a group regarded as a groupoid with single object. An element of 
$\pi_{1}(G)$ can be written as $(g_{1},g_{2},\cdots, g_{n})$ where $g_{i} \in G$ and $1 \leqq i \leqq n$.
The relations in Definition \ref{pi} imply that 
$$(g_{1},g_{2},\cdots, g_{n}) =  g_{1}g^{-1}_{2} \cdots g_{n}^{(-1)^{n-1}} $$ in $\pi_{1}(G)$.
It follows that $\pi_{1}(G) \cong G$.
}
\end{example}


\subsection{Coverings of small categories}

The notion of coverings is already defined 
in the category of spaces, simplicial sets \cite{GZ67} and groupoids \cite{GZ67}, \cite{May99}. 

Now we define coverings in the category of small categories, and consider 
relations between them.

\begin{definition}
Let $M$ be a category and  let $i : A \longrightarrow B$ and $p : X \longrightarrow Y$ be morphisms of $M$. 
We say that $p$ has the right lifting property for $i$ if for every commutative diagram in $M$ of the following form
\[
\xymatrix{
A \ar[d]_{i} \ar[r]^{f} & X \ar[d]^{p} \\
B \ar[r]_{g} & Y }
\]
there is a morphism $h : B \longrightarrow X$ such that $h \circ i = f$ and $p \circ h=g$. 
If such $h$ exists uniquely, then we say that $p$ has the unique right lifting property for $i$.
Let $S$ be a set of morphism in $M$. A morphism which has the right lifting property for any morphism in $S$ is called 
an $S$-injection. Denote the set of $S$-injections by $S$-inj. 
\end{definition}

\begin{definition} \label{cove}
A functor $p : E \longrightarrow B$ is called a covering if 
it has the unique right lifting property for $J_{1}$.
A covering $p : E \longrightarrow B$ is called a universal cover if $E$ is simply connected and $B$ is connected.
\end{definition}

\begin{lemma}\label{[n]}
A functor $p : E \longrightarrow B$ is a covering if and only if 
$p$ has the unique right lifting property for the inclusions
$\ast \longrightarrow [n]$ for all $n > 0$.
\begin{proof}
Since $[n]=0 \longrightarrow 1 \longrightarrow \cdots \longrightarrow n$, 
then we repeat taking lifts of $i \longrightarrow i+1$ starting at the point of the image $* \longrightarrow [n]$.
\end{proof}
\end{lemma}

Next we construct a universal cover over a connected category using the Grothendieck construction \cite{Tho79}.

\begin{definition}
Let $I$ be a small category and $\mathbf{Set}$ the category of sets. 
The Grothendieck construction of a functor 
$F : I \longrightarrow \mathbf{Set}$ is a small category $Gr(F)$ defined as follows.
The set of objects of $Gr(F)$ consists of pairs $(i,x)$ of an object $i \in I_{0}$ and an element $x \in F(i)$.
And a morphism $(i,x) \longrightarrow (j,y)$ 
in $Gr(F)$ is a morphism $f : i \longrightarrow j$ in $I$ such that $F(f)(x)=y$.
It admits the canonical projection $Gr(F) \longrightarrow I$ given by $(i,x) \longrightarrow i$.
\end{definition}

\begin{definition}\label{univer}
Let $(C,*)$ be a pointed and a connected category, then the category $\widehat{C}$ is defined by the 
Grothendieck construction of 
$$\pi(C)(*,-) : C \longrightarrow \mathbf{Set}.$$
The canonical projection $T : \widehat{C} \longrightarrow C$ carries  
an object of $\widehat{C}$ formed
$$(* \longrightarrow c_{1} \longleftarrow c_{2} \longrightarrow \cdots \longleftarrow c_{n})$$
to the last object $c_{n}$.
\end{definition}

\begin{lemma}\label{cover}
The canonical projection $T : \widehat{C} \longrightarrow C$ is a covering.
\begin{proof}
Suppose we have the following commutative diagram
\[
\xymatrix{
\ast \ar[d]_{j_{1}} \ar[r]^{x} & \widehat{C} \ar[d]^{T} \\
[1] \ar[r]_{g} & C. }
\]
The above $x$ gives a class of zigzag sequence of $C$ 
and $g(0 \longrightarrow 1) : g(0)=x_{n} \longrightarrow g(1)$ where $x_{n}$ is the last object of $x$.
Define $h : [1] \longrightarrow \widehat{C}$ by $h(0)=x$, 
$h(1)=(g(0 \longrightarrow 1)) \circ x$ and $h(0 \longrightarrow 1) = g(0 \longrightarrow 1).$
It makes the above diagram commutative and it exists uniquely.
 Similarly, $T$ has the unique lifting property for $j_{1}^{op} : * \longrightarrow [1]^{op}$.
\end{proof}
\end{lemma}

\begin{proposition}\label{0}
The category $\widehat{C}$ is simply connected.
\begin{proof}
For an object 
$$(*)=(c_{*}^{0}) \stackrel{f_{1}}{\longrightarrow} (c_{*}^{1}) \stackrel{f_{2}}{\longleftarrow} (c_{*}^{2}) \longrightarrow \cdots \stackrel{f_{n}}{\longleftarrow} (c_{*}^{n})=(*)$$
in $\pi_{1}(\widehat{C})$,
it suffices to show that 
$$(* \stackrel{f_{1}}{\longrightarrow} T(c_{*}^{1}) \stackrel{f_{2}}{\longleftarrow} T(c_{*}^{2}) \longrightarrow \cdots \stackrel{f_{n}}{\longleftarrow} *) =1$$
in $\pi_{1}(C)$.
By iterating the following process, we obtain
\begin{equation}
\begin{split}
&(* \stackrel{f_{1}}{\longrightarrow} T(c_{*}^{1}) \stackrel{f_{2}}{\longleftarrow} T(c_{*}^{2}) \longrightarrow \cdots \stackrel{f_{n}}{\longleftarrow} *) \\ \notag
=&(* \stackrel{c_{*}^{1}}{\longrightarrow} T(c_{*}^{1}) \stackrel{f_{2}}{\longleftarrow} T(c_{*}^{2}) \longrightarrow \cdots \stackrel{f_{n}}{\longleftarrow} *) \\
=&(* \stackrel{c_{*}^{2}}{\longrightarrow} T(c_{*}^{2}) \longrightarrow \cdots \stackrel{f_{n}}{\longleftarrow} *) = \cdots = 1.\\
\end{split}
\end{equation}
\end{proof}
\end{proposition}

\begin{corollary}
The canonical projection $T : \widehat{C} \longrightarrow C$ is a universal cover.
\end{corollary}

We recall the definition of coverings in the category of simplicial sets and groupoids \cite{GZ67}, \cite{May99}.

\begin{definition}
A morphism  $p : E \longrightarrow B$ in $\mathbf{SSet}$ is called a covering 
if it has the unique right lifting property for the inclusions $\Delta[0] \longrightarrow \Delta[n]$, $n \geqq 0$.
\end{definition}

\begin{definition}
A morphism $p : E \longrightarrow B$ in $\mathbf{Grd}$ is called a covering 
if it has the unique right lifting property for $K$ in Notation \ref{notation}.
\end{definition}

\begin{proposition}[\cite{GZ67}]
Both $S : \mathbf{Space} \longrightarrow \mathbf{SSet}$ and $|-| : \mathbf{SSet} \longrightarrow \mathbf{Space}$ 
preserve coverings.
\end{proposition}

\begin{proposition}\label{covg}
Both $i : \mathbf{Grd} \longrightarrow \mathbf{Cat}$ and $\pi : \mathbf{Cat} \longrightarrow \mathbf{Grd}$ 
preserve coverings.
\begin{proof}
Since $i$ is right adjoint to $\pi$, it preserves the unique right lifting property, thus it preserves coverings.
Conversely, let $p : E \longrightarrow B$ be a covering in $\mathbf{Cat}$. Consider the following commutative diagram in $\mathbf{Grd}$
\[
\xymatrix{
\ast \ar[d] \ar[r]^{e} & \pi(E) \ar[d]^{\pi(p)} \\
S^{\infty} \ar[r]_{f} & \pi(B). }
\]
Let $s$ be the image of morphism $0 \longrightarrow 1$ in $S^{\infty}$ by $f$.
It is a zigzag sequence of morphisms of $B$ starting at $p \circ e(*)$.
Since $p$ is a covering, we can find lifts of morphisms appearing in $s$, uniquely.
It gives a functor $S^{\infty} \longrightarrow \pi(E)$ making the diagram commutative, therefore $\pi(p)$ is a covering. 
\end{proof}
\end{proposition}

The category of simplicial sets and the category of small categories 
are related by the nerve functor and the categorization functor in \cite{GZ67}.

\begin{definition}\label{Nc}
The nerve functor $N : \mathbf{Cat} \longrightarrow \mathbf{SSet}$ is defined by
$$N_{n}C=\mathbf{Cat}([n],C)$$ and 
$$d_{i}(f_{1},\cdots,f_{n}) =(f_{1},\cdots ,f_{i-1},f_{i+1} \circ f_{i},f_{i+2},\cdots,f_{n})$$
and 
$$s_{j}(f_{1},\cdots,f_{n})=(f_{1},\cdots,f_{j},1,f_{j+1},\cdots,f_{n}).$$
The categorization functor $c : \mathbf{SSet} \longrightarrow \mathbf{Cat}$ is defined as follows.
The set of objects $cX_{0}$ is $X_{0}$ and 
morphisms in $cX$ are freely generated by the set $X_{1}$ subject to relations
given by elements of $X_{2}$, namely, $x_{1} = x_{2}x_{0}$ in $cX$ if there exists a $2$-simplex
$x$ such that $d_{2}x = x_{2}$, $d_{0}x = x_{0}$ and $d_{1}x = x_{1}$.
\end{definition}

\begin{proposition}\label{cN} {\rm \cite{GZ67}} The pair of functors 
$$c : \mathbf{SSet} \Longleftrightarrow \mathbf{Cat} : N$$ 
is an adjoint pair, and $cN \cong 1_{\mathbf{Cat}}$.
\end{proposition}

\begin{proposition}
A functor $p$ is a covering in $\mathbf{Cat}$ if and only if 
$N(p)$ is a covering in $\mathbf{SSet}$.
\begin{proof}
Since $N$ is right adjoint to $c$, $N$ preserves the unique right lifting property. 
Therefore, Lemma \ref{[n]} implies that $N$ preserves coverings.
Conversely, let $N(p)$ be a covering, then $N(p)$ has the unique right lifting property for 
$$d_{0}^{*},d_{1}^{*} : \Delta[0] \longrightarrow \Delta[1].$$
Since $cN \cong 1_{\mathbf{Cat}}$,  $p$ has the unique right lifting property for $J_{1}$.
\end{proof}
\end{proposition}

Before we end  of this section, let us define a Galois-type correspondence between 
subgroups of $\pi_{1}(C)$ and covering spaces over $C$ for a connected category $C$.
In the case of groupoids, May proved the following \cite{May99}. 

\begin{theorem}[\cite{May99}]
For a connected groupoid $G$, let $\mathbf{Cov}_{\mathbf{Grd}}(G)$ be the category of connected coverings over $G$ in $\mathbf{Grd}$ 
and let $\mathbf{O}(\pi_{1}(G))$ be the category consisting of subgroups  of $\pi_{1}(G)$ as objects
 and subconjugacy relations as morphisms.
Then there exists an equivalence of categories between $\mathbf{Cov}_{\mathbf{Grd}}(G)$ and $\mathbf{O}(\pi_{1}(G))$.
\end{theorem}

\begin{proposition}
For a connected category $C$, let $\mathbf{Cov}_{\mathbf{Cat}}(C)$ be the category of connected coverings over $C$ in $\mathbf{Cat}$.
Then there is an equivalence of categories between $\mathbf{Cov}_{\mathbf{Cat}}(C)$ and $\mathbf{Cov}_{\mathbf{Grd}}(\pi C)$.
\begin{proof}
The groupoidification functor induces $\pi : \mathbf{Cov}_{\mathbf{Cat}}(C) \longrightarrow \mathbf{Cov}_{\mathbf{Grd}}(\pi C)$ 
by Proposition \ref{covg}. On the other hand, let $q$ be a covering in $\mathbf{Grd}$ over $\pi C$, 
the pullback of $q$ along the canonical functor $C \longrightarrow \pi(C)$ induces a covering in $\mathbf{Cat}$ over $C$.
This correspondence gives an inverse functor $\mathbf{Cov}_{\mathbf{Grd}}(\pi C) \longrightarrow \mathbf{Cov}_{\mathbf{Cat}}(C)$ of $\pi$.
\end{proof}
\end{proposition}

\begin{corollary}
For a connected category $C$, there is an equivalence of categories between $\mathbf{Cov}_{\mathbf{Cat}}(C)$ and $\mathbf{O}(\pi_{1}(C))$.
\end{corollary}


 \section{The $1$-type model structure on $\mathbf{Cat}$}
 
 Model categories, first introduced by Quillen in \cite{Qui67}, form the foundation of homotopy theory.
This is a framework to do homotopy theory in general categories.
In this section, we define a model structure on the category of small categories, 
which is closely related to the notion of coverings, fundamental groups and the groupoidification.

\subsection{The $1$-type model structure on $\mathbf{Cat}$}

\begin{definition}
Suppose $M$ is a category. A functorial factorization is an ordered pair 
$(\alpha,\beta)$ of functors $\mathrm{Mor}(M) \longrightarrow \mathrm{Mor}(M)$ such that 
$f=\beta(f) \circ \alpha(f)$ for all morphisms $f$ in $M$, 
where $\mathrm{Mor}(M)$ is the category of morphisms of $M$.
\end{definition}

\begin{definition}\label{modef}
A model structure on a category $M$ consists of three distinguish classes of morphisms closed under retracts and compositions, 
the weak equivalences $W$, 
the cofibrations $C$, 
and the fibrations $F$, and two functorial factorizations $(\alpha,\beta)$ and $(\gamma, \delta)$ satisfying the following properties.
\begin{enumerate}
\item If $f$ and $g$ are morphisms of $M$ such that $g \circ f$ is defined and two of $f,g$ and $g \circ f$ are 
weak equivalences, then so is the third.
\item Every morphism in $W \cap C$ has the right lifting property for $F$, 
and every morphism in $C$ has the right lifting property for $W \cap F$.
\item For any morphism $f$ in $M$, $\alpha(f) \in C$  , $\beta(f) \in W \cap F$, 
$\gamma(f) \in W \cap C$ and $\delta(f) \in F$.
\end{enumerate}
A morphism in $W \cap C$ is called a trivial cofibration, and a morphism in $W \cap F$ is called a trivial fibration, respectively.

A model category is a category $M$ closed under small limits and colimits together with a model structure on $M$.
\end{definition}

It tends to be quite difficult to prove that a category admits a model structure.
The axioms of model structure are always hard to check.
However, there exists a technique to construct a new model structure from 
another good model structure.

\begin{definition}
We say that a model category $M$ is cofibrantly generated if there exist sets $A$ and $B$ of morphisms such that
\begin{enumerate}
\item both $A$ and $B$ permit the small object argument \cite{Hir03},

\item $W \cap F= A\textrm{-inj}$ and $F= B\textrm{-inj}$.
\end{enumerate}
The above set $A$ is called a generating cofibrations, and $B$ is called a generating trivial cofibrations.
Moreover, we say that $M$ is combinatorial if it is cofibrantly generated and locally presentable \cite{KL01}.
\end{definition}

\begin{example}\label{model} Let us recall several known model structures on $\mathbf{Cat}$, $\mathbf{Grd}$ and $\mathbf{SSet}$.
\begin{itemize}
\item The Joyal-Tierney model structure on $\mathbf{Cat}$ is defined as follows \cite{JT91}, \cite{Rez00}.
\begin{enumerate}
\item A morphism is a weak equivalence if it is an equivalence of categories.
\item A morphism is a cofibration if it is injective on the set of objects.
\end{enumerate}
Let $\mathbf{Cat}_{JT}$ be the category of small categories equipped with the above model structure.
\item Also $\mathbf{Grd}$ has the Anderson model structure with the same weak equivalences and cofibrations as the Joyal-Tierney model structure \cite{And78}.
Let $\mathbf{Grd}_{A}$ be the category of groupoids equipped with the above model structure.
\item Thomason found another model structure on $\mathbf{Cat}$ in \cite{Tho80} such that a functor is  
a weak equivalence if and only if the induced map between classifying spaces is a weak homotopy equivalence in $\mathbf{Space}$.
\item The Thomason model structure is closely related to 
the Kan model structure on $\mathbf{SSet} $\cite{Qui68}, \cite{Hov99} as follows.
\begin{enumerate}
\item A morphism is a weak equivalence if its geometric realization is a weak homotopy equivalence in $\mathbf{Space}$.
\item A morphism is a fibration if it is a Kan fibration.
\end{enumerate}
Let $\mathbf{SSet}_{K}$ be the category of simplicial sets equipped with the above model structure.
\end{itemize}
\end{example}

\begin{theorem}[\cite{Lur09}]
If $M$ is a combinatorial simplicial left proper model category and 
$S$ is a set of morphisms. 
Then the left Bousfield localization of $M$ with respect to $S$ does exist as a left proper simplicial combinatorial model category.
\end{theorem}

\begin{example}\label{jt}
The model category $\mathbf{Cat}_{JT}$ admits the generating cofibrations $I$ and the trivial cofibrations $K$ in Notation \ref{notation}. 
Since $\mathbf{Sets}$ is locally presentable, $\mathbf{Cat}$ is so \cite{KL01}.
For a small category $C$, let $\mu(C)$ be the maximal groupoid contained in $C$.
The function complex $\mathrm{Hom}(C,D)=N \mu(D^{C})$ gives rise to 
a simplicial enrichment for $\mathbf{Cat}$ where $D^{C}$ is the functor category from $C$ to $D$ \cite{Rez00}.
Since all objects in $\mathbf{Cat}_{JT}$ are fibrant and cofibrant, $\mathbf{Cat}_{JT}$ is left proper and right proper.
Thus the category $\mathbf{Cat}_{JT}$ is a combinatorial simplicial left proper model category.
\end{example}

\begin{definition}
Denote the Bousfield localization of 
$\mathbf{Cat}_{JT}$ with respect to the inclusion
$\varphi : [1] \longrightarrow S^{\infty}$ by $\mathbf{Cat}_{1}$.
\end{definition}

The model category $\mathbf{Cat}_{1}$ is called the $1$-type model category. 
It has the $\varphi$-local equivalences as weak equivalences and 
the cofibrations in $\mathbf{Cat}_{JT}$ as cofibrations.
We will show that a functor is a $\varphi$-local equivalence if and only if it is a weak $1$-equivalence.

\begin{definition}
A functor $f : C \longrightarrow D$ is called a weak $1$-equivalence 
if the both induced maps $\pi_{0}(C) \longrightarrow \pi_{0}(D)$ and
$\pi_{1}(C,x) \longrightarrow \pi_{1}(D,f(x))$ are isomorphisms for all $x \in C_{0}$. 
\end{definition}

\begin{lemma}\label{grou}
Let $G$ be a groupoid, then the canonical inclusion  $G \longrightarrow \pi(G)$ is an isomorphism of categories.
\begin{proof}
The inverse functor $\pi(G) \longrightarrow G$ is given by the identity map on the set of objects, and 
$$(\cdot \stackrel{f_{1}}{\longrightarrow} \cdot \stackrel{f_{2}}{\longleftarrow} \cdot \stackrel{f_{3}}{\longrightarrow} \cdots \stackrel{f_{n}}{\longleftarrow} \cdot)
\mapsto f_{n}^{-1} \circ \cdots \circ f_{3} \circ f_{2}^{-1} \circ f_{1}$$
on the set of morphisms.
\end{proof}
\end{lemma}

\begin{proposition}\label{1-equ}
The canonical inclusion $C \longrightarrow \pi(C)$ is a weak $1$-equivalence for any small category $C$.
\begin{proof}
The induced map on the set of objects is the identity map since $C_{0}=\pi(C)_{0}$.  
By Lemma \ref{grou}, the functor induces an isomorphism $\pi(C) \longrightarrow \pi(\pi(C))$ of categories.
Thus $\pi_{1}(C,*) \longrightarrow \pi_{1}(\pi(C),*)$ is an isomorphism.
\end{proof}
\end{proposition}

\begin{lemma}\label{equi}
A functor $f : C \longrightarrow D$ is  a weak $1$-equivalence if and only if the functor  $\pi(f) : \pi(C) \longrightarrow \pi(D)$ is an equivalence of categories.
\begin{proof}
If $\pi(f)$ is an equivalence, then it is obvious that $f$ is a weak $1$-equivalence by Proposition \ref{1-equ}
and the following commutative diagram
\[
\xymatrix{
C \ar[d] \ar[r]^{f} & D \ar[d] \\
\pi(C) \ar[r]_{\pi(f)} & \pi(D). }
\]
Let $f : C \longrightarrow D$ be a weak $1$-equivalence. 
Since $\pi_{n}(BG,*)=0$ for any pointed groupoid $(G,*)$ and $n \geqq 2$, 
the induced map $\pi(f)_{*} : \pi_{n}(B\pi(C),*) \longrightarrow \pi_{n}(B\pi(D),*)$ 
is an isomorphism for all $n \geqq 0$. 
This is a weak equivalence on $\mathbf{Cat}$ with the Thomason model structure in Example \ref{model}.
A functor between groupoids is a weak equivalence in the Thomason model structure
if and only if  it is an equivalence of categories \cite{CGT04}.
Thus $\pi(f)$ is an equivalence of categories.
\end{proof}
\end{lemma}

\begin{definition}
Let $M$ be a cofibrantly generated simplicial left proper model category 
and let $j : A \longrightarrow B$ be a morphism in $M$. We say that
\begin{enumerate}
\item a fibrant object $W$ is $j$-local if the induced morphism between the homotopy function complexes
$$j^{*} : \mathrm{Map}(B,W) \longrightarrow \mathrm{Map}(A,W)$$
is a weak equivalence in $\mathbf{SSet}_{K}$,
\item a morphism $f : X \longrightarrow Y$ is a $j$-local equivalence if $f^{*} : \mathrm{Map}(Y,W) \longrightarrow \mathrm{Map}(X,W)$ 
is a weak equivalence in $\mathbf{SSet}_{K}$ for all $j$-local objects $W$.
\end{enumerate}
\end{definition}

\begin{lemma}\label{local}
A small category is 
a $\varphi$-local object if and only if it is a groupoid.
\begin{proof}
Since all objects in $\mathbf{Cat}_{JT}$ are cofibrant and fibrant, the homotopy function complex 
$\mathrm{Map}(X,Y)$ in $\mathbf{Cat}_{JT}$ is weakly equivalent to the function complex $\mathrm{Hom}(X,Y)$ in Example \ref{jt}.
If $G$ is a groupoid, 
$$\mathrm{Hom}([1],G) \cong N G^{[1]} \cong N G^{S^{\infty}}\cong \mathrm{Hom}(S^{\infty},G).$$
Therefore $\varphi^{*} : \mathrm{Hom}(S^{\infty},G) \longrightarrow \mathrm{Hom}([1],G)$ is a weak equivalence in $\mathbf{SSet}_{K}$.
Conversely, assume $G$ is $\varphi$-local. Since $\varphi$ is a cofibration in $\mathbf{Cat}_{JT}$, $\varphi^{*}$ is a trivial fibration. Thus
$$\varphi^{*} : \mathrm{Hom}(S^{\infty}, G)_{0} \longrightarrow \mathrm{Hom}([1],G)_{0}$$
 is surjective.
Therefore, the map $\varphi^{*} : (G^{S^{\infty}})_{0} \longrightarrow (G^{[1]})_{0}$ is surjective. 
Hence $G$ is a groupoid.
\end{proof}
\end{lemma}

\begin{corollary}\label{fibrant}
A small category is fibrant in $\mathbf{Cat}_{1}$ if and only if it is a groupoid.
\begin{proof}
Since an object is fibrant in the left Bousfield localization with respect to the map $\varphi : [1] \longrightarrow S^{\infty}$
if and only if it is $\varphi$-local \cite{Hir03}.
\end{proof}
\end{corollary}

\begin{proposition}\label{weak}
A functor $f : X \longrightarrow Y$ is a $\varphi$-local equivalence if and only if it is a weak $1$-equivalence.
\begin{proof}
The functor $f$ induces the map between function complexes
$$f^{*} : N(W^{Y}) \longrightarrow N(W^{X})$$
for a $\varphi$-local object $W$. The both categories $W^{X}\cong W^{\pi X}$ and $W^{Y} \cong W^{\pi Y}$ are groupoids since $W$ is so.
Suppose $f^{*}$ is a weak equivalence, then  
$(\pi f)^{*} : W^{\pi Y} \longrightarrow W^{\pi X}$ is an equivalence of categories.
Take $W=\pi X$, we obtain an inverse of functor $\pi f$.
Therefore $\pi f$ is also an equivalence of categories, hence $f$ is a weak $1$-equivalence.
Conversely, we can prove that $f^{*}$ is a weak equivalence if $f$ is a weak $1$-equivalence. 
\end{proof}
\end{proposition}

The notion of weak $1$-equivalence also exists in $\mathbf{Space}$ and $\mathbf{SSet}$.

\begin{definition} \label{1-equ'}
A morphism $f : X \longrightarrow Y$ in $\mathbf{Space}$ is called a weak $1$-equivalence 
if the both induced maps $\pi_{0}(X) \longrightarrow \pi_{0}(Y)$ and
$\pi_{1}(X,x) \longrightarrow \pi_{1}(Y,f(x))$ are isomorphisms for all $x \in X$. 
On the other hand, a morphism $f$ in $\mathbf{SSet}$ is called 
a weak $1$-equivalence if its geometric realization $|f|$ is a weak $1$-eqivalence in $\mathbf{Space}$. 
By Theorem \ref{min}, $f$ is a weak $1$-equivalence in $\mathbf{Cat}$ 
if and only if its nerve $Nf$ is a weak $1$-equivalence in $\mathbf{SSet}$.
\end{definition}

\begin{theorem}[\cite{DP95}]\label{sset1}  There exists a model structure on $\mathbf{SSet}$ 
with the following weak equivalences and fibrations.
\begin{enumerate}
\item A morphism is a weak equivalence if and only if it is a weak $1$-equivalence.
\item A morphism is a fibration if and only if it has the right lifting property for $J'$,
\end{enumerate}
where 
$$J'=\{\Lambda_{j}^{n} \longrightarrow \Delta[n] , \Lambda^{3}_{k} \longrightarrow \partial \Delta[3] \ |\ 0 < n \leqq 2 , 0 \leqq j \leqq n , 0 \leqq k \leqq 3 \}.$$
Furthermore, this is a cofibrantly generated model structure with generating cofibrations $I'$ and trivial cofibrations $J'$,
 where
$$I'= \{ \partial \Delta[n] \longrightarrow \Delta[n]\ |\ 0 \leqq n \leqq 2\}.$$
Let $\mathbf{SSet}_{1}$ be the category of simplicial sets equipped with the above model structure.
\end{theorem}


\subsection{Fibrations and coverings}

In this section, we characterize the fibrations in $\mathbf{Cat}_{1}$.
It is closely related to coverings in $\mathbf{Cat}$ and Kan fibrations.

\begin{definition}{\rm
A functor $F : C \longrightarrow D$ is called fibered in groupoids if the following two conditions are satisfied: 
\begin{enumerate}
\item For every object $x$ in $C$ and every morphism $f : y \longrightarrow F(x)$ in $D$, there
exists a morphism $g : x' \longrightarrow  x$ in $C$ such that $F(g) = f$. 
\item For every morphism $f : x' \longrightarrow  x''$ in $C$ and every object $x$ in $C$, the map
$$C(x,x') \longrightarrow C(x,x'') \times_{D(F(x),F(x''))} D(F(x),F(x'))$$
given by $g \mapsto (f \circ g, F(g))$ is bijective. 
Similarly, we can define the notion of cofibered in groupoids \cite{Lur09}. 
\end{enumerate}
}
\end{definition}

\begin{proposition}\label{fibered}
A functor $F : C \longrightarrow D$ is fibered and cofibered in groupoids if and only if 
it has the right lifting property for $J$.
\begin{proof}
The first condition of fibered and cofibered in groupoids corresponds to the lifting property for $J_{1}$.
The map of the second condition is surjective if and only if 
the functor $F$ has the lifting property for $J_{2}$. 
Finally, the map is injective if and only if 
the functor $F$ has the lifting property for $J_{3}$. 
\end{proof}
\end{proposition}

\begin{theorem}\label{them}
The $1$-type model category $\mathbf{Cat}_{1}$ consists of the following structure. 
If $f : X \longrightarrow Y$ is a functor, then 
\begin{enumerate}
\item $f$ is a weak equivalence if and only if it is a weak $1$-equivalence,
\item $f$ is a cofibration if and only if $f_{0} : X_{0} \longrightarrow Y_{0}$ is injective,
\item $f$ is a fibration if and only if it is fibered and cofibered in groupoids.
\end{enumerate}
\begin{proof}
On the weak equivalences and cofibrations,  they are shown by Proposition \ref{weak} and the 
definition of the left Bousfield localization. Let us consider the fibration in $\mathbf{Cat}_{1}$.
We can put a cofibrantly generated model structure on $\mathbf{Cat}$ from $\mathbf{SSet}_{1}$ 
using the pair of adjoint functors (see in \cite{Hir03})
$$c : \mathbf{SSet}_{1} \Longleftrightarrow \mathbf{Cat} : N.$$
In the induced model structure on $\mathbf{Cat}$, 
a functor is a weak equivalence if and only if it is a weak $1$-equivalence, 
and the set of generating cofibration is $c(I')$ and generating trivial cofibration is $c(J')$.
Comparing $c(I')$ with $I$ implies that $c(I')\textrm{-inj} = I\textrm{-inj}$.
It follows that the classes of cofibrations are equal to each other.
Thus the induced model structure on $\mathbf{Cat}$ by the pair of adjoint functors $(c,N)$ 
coincides with the $1$-type model structure.
Also we can see that $c(J')\textrm{-inj} = J\textrm{-inj}$.
By Proposition \ref{fibered}, a fibration in $\mathbf{Cat}_{1}$ coincides with a functor which is fibered and cofibered in groupoids.
\end{proof}
\end{theorem}

\begin{theorem}[\cite{Lur09}]\label{kan} 
A functor $p$ is a fibration in $\mathbf{Cat}_{1}$ if and only if $Np$ is a Kan fibration.
\end{theorem}

\begin{corollary}\label{gr}
A category $G$ is a groupoid if and only if $N(G)$ is a Kan complex.
\begin{proof}
By Proposition \ref{fibrant}, groupoids coincide with the fibrant objects 
in $\mathbf{Cat}_{1}$, and 
Kan complexes coincides with the fibrant objects in $\mathbf{SSet}_{K}$. 
Thus Theorem \ref{kan} implies that $G$ is a groupoid if and only if $N(G)$ is a Kan complex since $N$ preserves terminal objects.
\end{proof}
\end{corollary}

\begin{lemma}\label{covfib}
If $p : E \longrightarrow B$ is a covering, then $p$ is a fibration in $\mathbf{Cat}_{1}$.
\begin{proof}
By the definition of coverings, $p$ has the lifting property for $J_{1}$.
Suppose we have the following commutative diagram
\[
\xymatrix{
I_{2} \ar[d]_{j_{2}} \ar[r]^{f} & E \ar[d]^{p} \\
[2] \ar[r]_{g} & B. }
\]
We obtain a morphism 
$\alpha : f(0) \longrightarrow f(2)$ over $g(0 \longrightarrow 1) : g(0) \longrightarrow g(1)$ by the lifting property of $p$. And 
$p \left( f(2 \longrightarrow 1) \circ \alpha  \right) = g(0 \longrightarrow 2)$ implies that 
$f(2 \longrightarrow 1) \circ \alpha =f(0 \longrightarrow 1)$ by the unique lifting property, 
then $p$ has the lifting property for $J_{2}$.
The unique lifting property implies that $p$ has the lifting property for $J_{3}$, similarly.
\end{proof}
\end{lemma}

\begin{definition}
Let $f : X \longrightarrow Y$ be a functor, then the category $f^{-1}(y)$ \cite{Qui73} is defined as a subcategory of $X$ 
for $y \in Y_{0}$,
${f^{-1}(y)}_{0}=f^{-1}(y)$ and $f^{-1}(y)(a,b)=p^{-1}(1_{y})$. 
A category is called discrete if the set of morphisms consists of only identity morphisms.
\end{definition}

\begin{proposition}\label{discrete}
A functor $p : E \longrightarrow B$ is a covering if and only if $p$ is a fibration in $\mathbf{Cat}_{1}$ and the category of fiber 
$p^{-1}(b)$ is discrete for any $b \in B_{0}$.
\begin{proof}
Let $p : E \longrightarrow B$ be a covering, then $p$ is a fibration by Lemma \ref{covfib}, 
and every fiber has the only identity morphisms by the unique lifting property.
Conversely, let $p$ be a fibration with discrete fibers. 
Since $p$ is a fibration, $p$ has the lifting property for $J_{1}$. 
We will show that the uniqueness of the lifting.
For the following commutative diagram
\[
\xymatrix{
\ast \ar[d]_{} \ar[r]^{e} & E \ar[d]^{p} \\
[1] \ar[r]_{f} & B }
\]
 we assume that $g,h : [1] \longrightarrow E$ satisfy $p \circ g = p \circ h = f$ and $g(0)=h(0)=e(*)$.
The lifting property of $p$ for $J_{2}$ implies that 
there exists $w : g(1) \longrightarrow h(1)$ such that $w \circ g =h$ and $p \circ w = 1$. 
Then $w$ is a morphism in $p^{-1}(f(1))$. 
However, $p^{-1}(f(1))$ has only identity morphisms, thus $w=1$. Therefore, $g=h$.
\end{proof}
\end{proposition}

The functorial factorization in $\mathbf{Cat}_{1}$ is given by the small object argument.
However, the small object argument is too abstract and difficult.
Now, we define another functorial factorization on $\mathbf{Cat}_{1}$ 
which induces the groupoidification in Definition \ref{pi} and  universal covers in Definition \ref{univer}.

\begin{definition} \label{track}
For a functor $f : X \longrightarrow Y$, 
define the category $E_{f}$ as
$$(E_{f})_{0} = \{ (x,y_{*}) \in X_{0} \times \mathrm{Mor}(\pi(Y))_{0}\ |\ f(x)=y_{0}\}$$
and
$$E_{f}((x,y_{*}),(x',y_{*}')) = 
\{ (g_{*},g) \in \pi(X)(x,x') \times Y(y_{n},y'_{m}) \ |\ y'_{*} \circ f(g_{*})  = g \circ y_{*} \in \pi(Y) \}$$
where $y_{n}$ and $y_{m}'$ are the last objects of $y_{*},y_{*}'$, respectively.
When $X=*$, the category $E_{f}$ is precisely  
$\widehat{Y}$ in Definition \ref{univer}.
Define a functor $j : X \longrightarrow E_{f}$ by 
$x \mapsto (x,1_{f(x)})$ and $p : E_{f} \longrightarrow Y$ by 
$(x,y_{*}) \mapsto y_{n}$.
Define $\alpha, \beta : \mathrm{Mor}(\mathbf{Cat}) \longrightarrow \mathrm{Mor}(\mathbf{Cat})$ as $\alpha(f)=j$ and $\beta(f)=p$,
then $(\alpha,\beta)$ is a functorial factorization of $\mathbf{Cat}$.
\end{definition}

\begin{proposition}\label{propa}
The functor $p : E_{f} \longrightarrow Y$ is a fibration in $\mathbf{Cat}_{1}$.
\begin{proof}
Suppose we have the following commutative diagram
\[
\xymatrix{
\ast \ar[d]_{j_{1}} \ar[r]^{\alpha} & E_{f} \ar[d]^{p} \\
[1] \ar[r]_{\beta} & Y. }
\]
Let $\alpha(*)=(x,y_{*})$, then $\beta(0 \longrightarrow 1) : \beta(0)=y_{n} \longrightarrow \beta(1)$.
Define a functor $\gamma : [1] \longrightarrow E_{f}$ by 
$\gamma(0)=\alpha(*)=(x,y_{*})$, $\gamma(1)=(x, \beta(0 \longrightarrow 1) \circ y_{*} )$ and $\gamma(0 \longrightarrow 1) =(1, \beta(0 \longrightarrow 1)).$
It makes the above diagram commutative, then $p$ has the lifting property for $J_{1}$. 

Suppose we have the following commutative diagram
\[
\xymatrix{
I_{2} \ar[d]_{j_{2}} \ar[r]^{\alpha} & E_{f} \ar[d]^{p} \\
[2] \ar[r]_{\beta} & Y, }
\]
the image of $\alpha$ describes the diagram in $E_{f}$ as
$$(x',y_{*}') \stackrel{(g_{*},g)}{\longrightarrow} (x,y_{*}) \stackrel{(h_{*},h)}{\longleftarrow} (x'',y_{*}'').$$
Now $\beta(1 \longrightarrow 2)$ is a morphism from the last object of $y'_{*}$ to $y_{*}''$.
Thus the morphism
$$(h_{*} \circ g_{*}^{-1},\beta(1 \longrightarrow 2)) : (x',y_{*}') \longrightarrow (x'',y_{*}'')$$
gives a functor $[2] \longrightarrow E_{f}$ making the above diagram commutative, then $p$ has the lifting property for $J_{2}$.

Suppose we have the following commutative diagram
\[
\xymatrix{
CS^{1} \ar[d]_{j_{3}} \ar[r]^{\alpha} & E_{f} \ar[d]^{p} \\
[2] \ar[r]_{\beta} & Y, }
\]
the image of $\alpha$ describes the diagram in $E_{f}$ as
\[
\xymatrix
{
(x,y_{*}) \ar@<0ex>[r]^{ (g_{*},g) } & (x',y_{*}') \ar@<1ex>[r]^{(h_{*},h)} \ar@<-1ex>[r]_{(h_{*}',h')} & (x'',y_{*}''). \\
}
\]
Since $ h_{*} \circ g_{*} = h'_{*} \circ g_{*} $ in $\pi(X)$, then 
$$h_{*}=h_{*} \circ g_{*} \circ g_{*}^{-1}=h'_{*} \circ g_{*} \circ g_{*}^{-1}=h'_{*}$$ 
in $\pi(X)$. Moreover, $\beta$ implies that $h=h'$, then $(h_{*},h)=(h_{*}',h')$ 
 and it gives a functor $[2] \longrightarrow E_{f}$ making the above diagram commutative, then $p$ has the lifting property for $J_{3}$.
\end{proof}
\end{proposition}

\begin{proposition}\label{propb}
The functor $j : X \longrightarrow E_{f}$ is a trivial cofibration in $\mathbf{Cat}_{1}$.
\begin{proof}
It is obvious that $j$ is a cofibration, thus it saffices to prove that $j$ is a weak $1$-equivalence.
We will show that $j_{*} : \pi_{0}(X) \longrightarrow \pi_{0}(E_{f})$ is an isomorphism.

We take an element $[x,y_{*}] \in \pi_{0}(E_{f})$. Suppose $(x,y_{*})$ is described by the diagram
$$f(x)=y_{0} \longrightarrow y_{1} \longleftarrow y_{2} \longrightarrow \cdots \longleftarrow y_{n}.$$
Let us consider the next commutative diagram
\[
\xymatrix
{
f(x) \ar[dr] \ar[r]^{=} & f(x) \ar[d] & f(x) \ar[d] \ar[l]_{=} \ar[r]^{=} & \cdots \\
& y_{1} & y_{1} \ar[l]_{=} \ar[r]^{=} & \cdots \\
&& y_{2} \ar[u] \ar[r] \ar[ul] \ar[dr] & \cdots \\
&&& \cdots .
}
\]

When we regard vertical sequences as 
objects in $E_{f}$, the above diagram implies that 
$$j_{*}[x]=[x,1_{f(x)}]=[x,y_{*}].$$
Thus $j_{*}$ is surjective.
On the other hand, for $[x],[y] \in \pi_{0}(X)$, assume that $j_{*}[x]=j_{*}[y]$ in $\pi_{0}(E_{f})$. 
It is obvious that $[x]=[y] \in \pi_{0}(X)$ 
from the definition of the set of morphisms of $E_{f}$. Thus $j_{*}$ is injective.

We will show that the induced map $j_{*} : \pi_{1}(X,x) \longrightarrow \pi_{1}(E_{f},(x,1_{f(x)}))$ is an 
isomorphism for any $x \in X_{0}$.
We take an element
$$\sigma : (x,1_{f(x)}) \stackrel{(g_{1,*},g_{1})}{\longrightarrow} (x_{1},y_{1,*}) \stackrel{(g_{2,*},g_{2})}{\longleftarrow} (x_{2},y_{2,*}) 
\longrightarrow \cdots \stackrel{(g_{n,*},g_{n})}{\longleftarrow} (x_{n},y_{n,*})=(x,1_{f(x)})$$
in $\pi_{1}(X,(x,1_{f(x)}))$. For the 
$i$-th morphism $(g_{i,*},g_{i}) : (x_{i-1},y_{i-1,*}) \longrightarrow (x_{i},y_{i,*})$ (in the case of $i$ odd), 
the morphism $g_{i,*}$ is an element of $\pi(X)(x_{i-1},x_{i})$.
Let 
$$\rho=g_{n,*}^{-1} \circ \cdots \circ g_{2,*}^{-1} \circ g_{1,*} \in \pi_{1}(X,x),$$
then $j_{*}(\rho)=\sigma$ in $\pi_{1}(E_{f})$. 
Therefore $j_{*}$ is surjective. 
Finally, 
$$\textrm{Ker }(j_{*})=\{x_{*} \in \pi_{1}(X,x)\ |\ j_{*}(x_{*})=0 \in \pi_{1}(E_{f},(x,1_{f(x)})\}=0$$
since the both morphisms and compositions in $E_{f}$ are same as $X$. 
Thus $j_{*}$ is injective.
\end{proof}
\end{proposition}

\begin{corollary}
The functorial factorization $(\alpha, \beta)$ in Definition \ref{track} 
satisfies the third axiom of model structure in Definition \ref{modef}. 
\end{corollary}

\begin{corollary}\label{poi}
Let $C$ be a category, and denote by $q : C \longrightarrow *$ the morphism 
from $C$ to the terminal object $*$ in $\mathbf{Cat}$.
Then $\alpha(q)$ is the functor $C \longrightarrow \pi(C)$ in Definition \ref{pi},
where $(\alpha,\beta)$ is the functorial factorization in Definition \ref{track}.
\end{corollary}

\begin{corollary}\label{un}
Let $(C,*)$ be a pointed connected category, and let $k : * \longrightarrow C$ 
be the embedding functor to the base point.
Then $\beta(k)$ is the universal cover $\widehat{C} \longrightarrow C$ in Definition \ref{univer}, 
where $(\alpha,\beta)$ is the functorial factorization in Definition \ref{track}.
\end{corollary}


\subsection{Relations between the $1$-type model category and other model categories}

The model category $\mathbf{Cat}_{1}$ is related to other model categories by the following 
pairs of adjoint functors
\[ 
\xymatrix{
\mathbf{Grd} \ar@<1ex>[r]^{i} & \mathbf{Cat}  \ar@<1ex>[l]^{\pi} \ar@<1ex>[r]^{N}   &\mathbf{SSet} \ar@<1ex>[r]^{|-|} \ar@<1ex>[l]^{c}  & \mathbf{Space} \ar@<1ex>[l]^{S} 
}
\]

\begin{definition}
\label{quiequ}
Let $M$ and $N$ be model categories and let 
$$F : M \Longleftrightarrow N : G$$
be a pair of adjoint functors. We say that $(F,G)$ is a Quillen pair if $F$ preserves cofibrations and 
$G$ preserves fibrations.
Furthermore $(F,G)$ is called a pair of Quillen equivalences 
if for every cofibrant object $X$ in $M$, every fibrant object $Y$ in $N$, and every map 
$f : X \longrightarrow GY$ in $M$, the map $f$ is a weak equivalence in $M$ if and only if the adjoint map 
$f^{\sharp} : FX \longrightarrow Y$ is a weak equivalence in $N$.
\end{definition}

\begin{proposition}[\cite{Hir03}]\label{hir}
Let $M$ and $N$ be model categories and let 
$$F : M \Longleftrightarrow N : G$$ 
be a pair of adjoint functors. Then the following are equivalent:
\begin{enumerate}
\item $(F,G)$ is a Quillen pair.
\item $F$ preserves both cofibrations and trivial cofibrations.
\item $G$ preserves both fibrations and trivial fibrations.
\end{enumerate}
\end{proposition}

\begin{proposition}The pair of adjoint functors
$$\pi : \mathbf{Cat}_{1} \Longleftrightarrow \mathbf{Grd}_{A} : i$$
is a pair of Quillen equivalences.
\begin{proof}
Since $\pi$ preserves weak equivalences and cofibrations,  Proposition \ref{hir} implies that 
$(\pi,i)$ is a Quillen pair.
Furthermore, $(\pi,i)$ is a pair of Quillen equivalences since the canonical inclusion 
$C \longrightarrow \pi(C)$ is a weak $1$-equivalence for any small category $C$.
\end{proof}
\end{proposition}

\begin{proposition}\label{kan-1}
The pair of adjoint functors
$$c : \mathbf{SSet}_{K} \Longleftrightarrow \mathbf{Cat}_{1} : N$$
is a Quillen pair.
\begin{proof}
Theorem \ref{kan} implies that $N$ preserves fibrations. 
Since a cofibration $i : A \longrightarrow B$ in $\mathbf{SSet}_{K}$ is injective for all dimensions, 
in particular, $i_{0} : A_{0} \longrightarrow B_{0}$ is injective. The map on the set of objects of 
$ci : cA \longrightarrow cB$ coincides with $i_{0} : A_{0} \longrightarrow B_{0}$. Thus $ci$ is a cofibration.
Therefore $c$ preserves cofibrations, and $(c,N)$ is a Quillen pair.

\end{proof}
\end{proposition}

We will prove that $\mathbf{Cat}_{1}$ is Quillen equivalent to $\mathbf{SSet}_{1}$.

\begin{definition} \label{homoto}
Let $(X,*)$ be a pointed Kan complex. Two $1$-simplices $x,y \in X_{1}$ satisfying $d_{i}(x)=d_{i}(y)=*$ for $i=0,1$
are called homotopic, denoted by $x \simeq y$, if there exists a 
$2$-simplex $z \in X_{2}$ such that 
$d_{0}z=*$, $d_{1}z=x$ and $d_{2}z=y$.
The fundamental group $\pi_{1}(X,*)$ is defined by 
$$\{x \in X_{1} \ |\ d_{i}(x)=*, i=0,1\}/ \simeq .$$
\end{definition}

\begin{lemma}[\cite{May92}] \label{may}
There exists a group structure on $\pi_{1}(X,*)$ under which 
$$\theta_{*} : \pi_{1}(X,*) \longrightarrow \pi_{1}(S|X|,*) = \pi_{1}(|X|,*)$$
is an isomorphism of groups for a pointed Kan complex $(X,*)$, where $\theta : X \longrightarrow S|X|$ is the counit map of the pair of adjoint functors $(|-|,S)$.
\end{lemma}

\begin{lemma}\label{gro}
If $X$ is a Kan complex, then $cX$ is a groupoid.
\begin{proof}
A morphism of $cX$ from $a$ to $b$ is a class of sequence $e_{1}e_{2} \cdots e_{n}$ of 1-simplexes of $X$.
There exists $e \in X_{1}$ satisfying $e= e_{1}e_{2}\cdots e_{n}$ in $cX$ since $X$ is a Kan complex.
Furthermore, there exists $d \in X_{1}$ such that $de = a$ and $ed=b$ ,
thus all morphisms of $cX$ are invertible.
\end{proof}
\end{lemma}

\begin{proposition}\label{eta-equ}
The counit map $\eta : X \longrightarrow NcX$ is a weak $1$-equivalence in $\mathbf{SSet}$ if $X$ is a Kan complex.
\begin{proof}
It is obvious that $\eta_{*} : \pi_{0}(X) \longrightarrow \pi_{0}(NcX)$ is the identity map 
because $X_{0}=(NcX)_{0}$.
Now, $cX$ is a groupoid by Lemma \ref{gro} then 
$$\pi_{1}(NcX,x) \cong cX(x,x) = \pi_{1}(X,x).$$
Thus $\eta_{*} : \pi_{1}(X,*) \longrightarrow \pi_{1}(NcX,*)$ is an 
isomorphism.
\end{proof}
\end{proposition}

\begin{corollary}\label{eta-equ'}
The counit map $\eta : X \longrightarrow NcX$ is a weak $1$-equivalence in $\mathbf{SSet}$ for any $X$.
\begin{proof}
By the functorial factorization in $\mathbf{SSet}_{K}$, 
there exists a Kan complex $RX$ and a trivial cofibration $i : X \longrightarrow RX$ in 
$\mathbf{SSet}_{K}$ for $X$. The following diagram
\[
\xymatrix{
X \ar[d]_{i} \ar[r]^{\eta} & NcX \ar[d]^{Nci} \\
RX \ar[r]_{\eta} & NcRX }
\]
is commutative, and $\eta : RX \longrightarrow NcRX$ is 
a weak $1$-equivalence by Proposition \ref{eta-equ} and also $i$ is a weak $1$-equivalence. 
By Proposition \ref{kan-1},
$$c : \mathbf{SSet}_{K} \Longleftrightarrow \mathbf{Cat}_{1} : N$$
is a Quillen pair, therefore $c$ preserves trivial cofibrations by Proposition \ref{hir}.
Then $ci$ is a trivial cofibration in $\mathbf{Cat}_{1}$, in particular, 
$ci$ is a weak $1$-equivalence in $\mathbf{Cat}$. 
Therefore $Nci$ is a weak $1$-equivalence in $\mathbf{SSet}$, and $\eta : X \longrightarrow NcX$ 
is a weak $1$-equivalence.
\end{proof}
\end{corollary}

\begin{theorem}
The pair of adjoint functors
$$c : \mathbf{SSet}_{1} \Longleftrightarrow \mathbf{Cat}_{1} : N$$
is a pair of Quillen equivalences.
\begin{proof}
By Theorem \ref{them}, $(c,N)$ is a Quillen pair.
Suppose $X$ is a cofibrant object in $\mathbf{SSet}_{1}$ and $G$ is a fibrant object in $\mathbf{Cat}_{1}$ 
and $f : X \longrightarrow NG$ is a weak equivalence in $\mathbf{SSet}_{1}$.
The map $f^{\sharp} : cX \longrightarrow G$ is given by  $cX \stackrel{cf}{\longrightarrow} cNG \cong G$.
Now the following diagram
\[
\xymatrix{
X \ar[d]_{\eta} \ar[r]^{f} & NG \ar[d]^{\cong} \\
NcX \ar[r]_{Ncf} & NcNG}
\]
is commutative, then $Ncf$ is a weak equivalence since 
$\eta$ 
is a weak equivalence from Corollary \ref{eta-equ'}. Thus $cf$ is a weak equivalence in $\mathbf{Cat}_{1}$.
Conversely, it is obvious that $f$ is a 
weak equivalence in $\mathbf{SSet}_{1}$ if $f^{\sharp}$ is a weak equivalence in $\mathbf{Cat}_{1}$.
\end{proof}
\end{theorem}

\section*{Acknowledgement}
I would like to thank Professor Dai Tamaki and Katsuhiko Kuribayashi for helpful advices about this paper. 
Their several suggestions and remarks were essential in the development and revisions of this work.
I am also grateful to Professor Hideto Asashiba for conversations. His view points and suggestions really helped me.

\end{document}